\documentclass[11pt]{amsart}
\usepackage[T1]{fontenc}
\usepackage[english]{babel}
\usepackage{amscd,amsmath,amsthm,amssymb,graphics}
\usepackage{lmodern,pst-node}
\usepackage{pstcol,pst-plot,pst-3d}
\usepackage{multicol}
\usepackage{epic,eepic}
\usepackage{amsfonts,amssymb,amscd,amsmath,enumitem,verbatim}
\psset{unit=0.7cm,linewidth=0.8pt,arrowsize=2.5pt 4}

\newpsstyle{fatline}{linewidth=1.5pt}
\newpsstyle{fyp}{fillstyle=solid,fillcolor=verylight}
\definecolor{verylight}{gray}{0.97}
\definecolor{light}{gray}{0.9}
\definecolor{medium}{gray}{0.85}
\definecolor{dark}{gray}{0.6}

\def\frk{\frak}               

\def\Phi{{\frk n}}
\def\Phi{{\frk N}}
\def\opn#1#2{\def#1{\operatorname{#2}}} 
\opn\chara{char} \opn\length{\ell} \opn\pd{pd} \opn\rk{rk}
\opn\projdim{proj\,dim} \opn\injdim{inj\,dim} \opn\rank{rank}
\opn\depth{depth} \opn\grade{grade} \opn\height{height}
\opn\embdim{emb\,dim} \opn\codim{codim}

\opn\Tr{Tr} \opn\bigrank{big\,rank}
\opn\superheight{superheight}\opn\lcm{lcm}
\opn\trdeg{tr\,deg}
\opn\reg{reg} \opn\lreg{lreg} \opn\ini{in} \opn\lpd{lpd}
\opn\size{size}\opn\bigsize{bigsize}
\opn\cosize{cosize}\opn\bigcosize{bigcosize}
\opn\sdepth{sdepth}\opn\sreg{sreg}
\opn\link{link}\opn\fdepth{fdepth}
\opn\deg{deg}
\opn\max{max}
\opn\indeg{indeg}
\opn\min{min}
\opn\psln{psln}
\opn\div{div} \opn\Div{Div} \opn\cl{cl} \opn\Cl{Cl}
\let\epsilon\varepsilon
\let\phi=\varphi
\let\kappa=\varkappa

\opn\Spec{Spec} \opn\Supp{Supp} \opn\supp{supp} \opn\Sing{Sing}
\opn\Ass{Ass} \opn\Min{Min}\opn\Mon{Mon} \opn\dstab{dstab} \opn\astab{astab}
\opn\Syz{Syz}
\opn\Ann{Ann} \opn\Rad{Rad} \opn\Soc{Soc}
\opn\Im{Im} \opn\Ker{Ker} \opn\Coker{Coker} \opn\Am{Am}
\opn\Hom{Hom} \opn\Tor{Tor} \opn\Ext{Ext} \opn\End{End}
\opn\Aut{Aut} \opn\id{id}

\opn\nat{nat}
\opn\pff{pf}
\opn\Pf{Pf} \opn\GL{GL} \opn\SL{SL} \opn\mod{mod} \opn\ord{ord}
\opn\Gin{Gin} \opn\Hilb{Hilb}\opn\sort{sort}
\opn\initial{init}
\opn\ende{end}
\opn\height{height}
\opn\depth{depth}
\opn\type{type}
\opn\ldim{ldim}

\opn\aff{aff} \opn\con{conv} \opn\relint{relint} \opn\st{st}
\opn\lk{lk} \opn\cn{cn} \opn\core{core} \opn\vol{vol}
\opn\link{link} \opn\star{star}\opn\lex{lex}
\opn\gr{gr}

\def\pot#1#2{#1[\kern-0.28ex[#2]\kern-0.28ex]}

\opn\dirlim{\underrightarrow{\lim}}
\opn\inivlim{\underleftarrow{\lim}}
\def\Implies{\ifmmode\Longrightarrow \else
        \unskip${}\Longrightarrow{}$\ignorespaces\fi}
\def\implies{\ifmmode\Rightarrow \else
        \unskip${}\Rightarrow{}$\ignorespaces\fi}
\def\iff{\ifmmode\Longleftrightarrow \else
        \unskip${}\Longleftrightarrow{}$\ignorespaces\fi}

\let\:=\colon
 \theoremstyle{plain}
\newtheorem{Theorem}{Theorem}[section]
 \newtheorem{Lemma}[Theorem]{Lemma}
 
 \newtheorem{Proposition}[Theorem]{Proposition}

 \newtheorem{Question}[Theorem]{Question}
 \theoremstyle{definition}
 \newtheorem{Definition}[Theorem]{Definition}

 \newtheorem{Example}[Theorem]{Example}

\let\epsilon\varepsilon
\let\kappa=\varkappa
\def\qed{\ifhmode\textqed\fi
      \ifmmode\ifinner\quad\qedsymbol\else\dispqed\fi\fi}
\def\textqed{\unskip\nobreak\penalty50
       \hskip2em\hbox{}\nobreak\hfil\qedsymbol
       \parfillskip=0pt \finalhyphendemerits=0}
\def\dispqed{\rlap{\qquad\qedsymbol}}

\opn\dis{dis}
\def\pnt{{\raise0.5mm\hbox{\large\bf.}}}

\opn\Lex{Lex}

\begin{document}

\author[Mafi and Naderi]{ Amir Mafi and Dler Naderi}
\title{A note on linear resolution and polymatroidal ideals}

\address{Amir Mafi, Department of Mathematics, University Of Kurdistan, P.O. Box: 416, Sanandaj, Iran.}
\email{A\_Mafi@ipm.ir}
\address{Dler Naderi, Department of Mathematics, University of Kurdistan, P.O. Box: 416, Sanandaj,
Iran.}
\email{dler.naderi65@gmail.com}

\begin{abstract}
Let $R=K[x_1,...,x_n]$ be the polynomial ring in $n$ variables over a field $K$ and $I$ be a monomial ideal generated in degree $d$.
Bandari and Herzog conjectured that a monomial ideal $I$ is polymatroidal if and only if all its monomial localizations have a linear resolution.
In this paper we give an affirmative answer to the conjecture in the following cases: $(i)$ $\height(I)=n-1$; $(ii)$ $I$ contains at least $n-3$ pure powers of the variables $x_1^d,...,x_{n-3}^d$; $(iii)$ $I$ is a monomial ideal in at most four variables.
\end{abstract}


\maketitle
\section*{Introduction}
Throughout this paper we assume that $R=K[x_1,...,x_n]$ is the polynomial ring in $n$ variables over a field $K$, $\frak{m}=(x_1,...,x_n)$ the unique homogeneous maximal ideal, and $I$ a monomial ideal. We denote by $G(I)$ the unique minimal set of monomial generators of $I$. A polymatroidal ideal is a monomial ideal, generated in a single degree, satisfying the following conditions: For all monomials $u,v\in G(I)$ with $\deg_{x_i}(u)>\deg_{x_i}(v)$, there exists an index $j$ such that $\deg_{x_j}(v)>\deg_{x_j}(u)$ and $x_j(u/x_i)\in I$ (see \cite{HH} or \cite{HH1}). A squarefree polymatroidal ideal is called a matroidal ideal.

Three general properties of polymatroidal ideals are crucial: $(i)$ all powers of polymatroidal ideals are again polymatroidal \cite[Theorem 5.3]{CH}, $(ii)$
polymatroidal ideals have linear quotients \cite[Lemma 1.3]{HT}, which implies that they have linear resolutions, $(iii)$ localizations of polymatroidal ideals at monomial prime ideals are again polymatroidal \cite[Corollary 3.2]{HRV}. The monomial localization of a monomial ideal $I$ with respect to a monomial prime ideal $\frak{p}$ is the monomial ideal $I(\frak{p})$ which obtained from $I$ by substituting the variables $x_i\notin\frak{p}$ by $1$. The monomial localization $I(\frak{p})$  can also be described as the saturation $I:(\prod_{x_i\notin{\frak{p}}}x_i)^{\infty}$ and when $I$ is a squarefree monomial ideal we see that $I(\frak{p})=I:(\prod_{x_i\notin{\frak{p}}}x_i)$. Also, we denote by $R(\frak{p})$ the polynomial ring over $K$ in the variables which belong to $\frak{p}$ and we denote $\frak{m}_{\frak{p}}$ the graded maximal ideal of $R(\frak{p})$.

Bandari and Herzog \cite[Conjecture 2.9]{BH} conjectured that the monomial ideals with the property that all monomial localizations have a linear resolution are precisely the polymatroidal ideal. They gave an affirmative answer to the conjecture in the following cases: $(i)$ $I$ is a squarefree monomial ideal; $(ii)$ $I$ is generated in degree $2$; $(iii)$ $I$ contains at least $n-1$ pure powers; $(iv)$ $I$ is a monomial ideal in at most three variables; $(v)$ $I$ has no embedded prime ideal and $\height(I)=n-1$.

Herzog, Hibi and Zheng \cite[Theorem 3.2]{HHZ} proved that if $I$ is a monomial ideal generated in degree $2$, then $I$ has a linear resolution if and only if each power of $I$ has a linear resolution. Sturmfels \cite{S} gave an example $I=(def,cef,cdf,cde,bef,bcd,acf,ade)$ with $I$ has a linear resolution while $I^2$ has no linear resolution (see also \cite{C}). This suggests the following question: Is it true that each power of $I$  has a linear resolution, if $I$ is a squarefree monomial ideal of degree $d$ with $I^k$ has a linear resolution for all $1\leq k\leq d-1$?

In this paper we give a counterexample to this question and also we give an affirmative answer to Bandari-Herzog's conjecture in the following cases:
$(i)$  $\height(I)=n-1$; (ii) $I$ contains at least $n-3$ pure powers of the variables $x_1^d,...,x_{n-3}^d$; (iii) $I$ is monomial ideal in at most four variables.

For any unexplained notion or terminology, we refer the reader to \cite{HH1}.
Several explicit examples were  performed with help of the computer algebra systems Macaulay2 \cite{GS}.

\section{Monomial localizations of polymatroidal ideals}
It is known that if $I\subset R$ is an $\frak{m}$-primary monomial ideal which has a linear resolution, then there exists a positive integer $k$ such that
$I={\frak{m}}^k$. Also, it is known that $I=uJ$, where $u$ is the greatest common divisor of the generator of $I$, is a linear resolution if and only if $J$ a linear resolution. Hence if $I\subset R=K[x_1,x_2]$ is a linear resolution, then $I$ is principal ideal or $I=uJ$. If $I=uJ$, then $J$ has a linear resolution and $\frak{m}$-primary monomial ideal. Thus if $I$ is principal or $I=uJ$, then $I$ is a polymatroidal ideal. Hence $I\subset R=K[x_1,x_2]$ is polymatroidal if and only if $I$ is a linear resolution.

For a monomial ideal $I$ of $R$ and $G(I)=\{u_1,...,u_t\}$, we set $\supp(I)=\cup_{i=1}^t\supp(u_i)$, where $\supp(u)=\{x_i: u=x_1^{a_1}...x_n^{a_n}, a_i\neq 0\}$ and we set $\gcd(I)=\gcd(u_1,...,u_m)$. We say that the monomial ideal $I$ is full-supported if $\supp(I)=\{x_1,...,x_n\}$.

In the sequel we recall the following definitions from \cite[Definition 8]{T}, \cite[Definition 2.5]{KM} and \cite[Definition 2.3]{BH}.

\begin{Definition}
Let $I\subset R$ be a monomial ideal and let $x^t[i]$ denote the monomial $x_1^{t_1}...\widehat{x_i^{t_i}}...x_n^{t_n}$, where the term $x_i^{t_i}$ of $x^t$ is omitted. For each $i=1,...,n$, we put $I[i]=(I:x_i^{\infty})$.
\end{Definition}

\begin{Definition}
Let $d, a_1,..., a_n$ be positive integers. We put $I_{(d;a_1,...,a_n)}\subset R$ be the monomial ideal generated by the monomials $u\in R$ of degree $d$ satisfying $\deg_{x_i}(u)\leq a_i$ for all $i=1,...,n$. Monomial ideals of this type are called ideals of Veronese type.
Monomial ideals of Veronese type are polymatroidal.
\end{Definition}

\begin{Lemma}\label{L}
Let $I\subset R=K[x_1,x_2,x_3]$ be a squarefree monomial ideal. Then $I$ has a linear resolution if and only if $I$ is a matroidal ideal.
\end{Lemma}
\begin{proof}
If $I$ is a matroidal ideal, then it is clear $I$ has a linear resolution.
Conversely, suppose $I$ has a linear resolution and as above explanation we can assume that $\gcd(I)=1$. Therefore $\height(I)\geq 2$ and so $\height(I[i])\geq 2$ for $i=1,2,3$. Thus $I[1]=(x_2,x_3)$, $I[2]=(x_1,x_3)$ and $I[3]=(x_1,x_2)$ and so by \cite[Proposition 2.11]{BH}, I is a squarefree Veronese type ideal.
\end{proof}

The following example shows that the squarefree condition over $I$ in Lemma \ref{L} is essential.
For its proof we use the following well-known fact: a monomial ideal $I$ of degree $d$ is a linear resolution if and only if $\reg(I)=d$.
\begin{Example}
Let $I=(x_1x_2,x_1x_3,x_2^2)$ be an ideal of $R=k[x_1,x_2,x_3]$. Then $\reg(I)=2$ and so $I$ has a linear resolution but $I$ is not polymatroidal.
\end{Example}

The following result was proved in \cite[Proposition 2.7]{BH}. We reprove with a simplified proof.
\begin{Proposition}
Let $I\subset R=K[x_1,x_2,x_3]$ be a monomial ideal. Then $I$ is polymatroidal if and only if $I$ and $I[i]$ have a linear resolution for $i=1,2,3$.
\end{Proposition}

\begin{proof}
If $I$ is a polymatroidal ideal, then by \cite[Theorem 1.1]{BH} $I[i]$ is polymatroidal for $i=1,2,3$ and so they have a linear resolution.
Conversely, we can assume that $\gcd(I)=1$ and so $\height(I)\geq 2$. Thus $\height(I[i])\geq 2$ and so $I[i]$ is a Veronese ideal for $i=1,2,3$ in polynomial ring with two variables. Therefore by \cite[Proposition 2.11]{BH} $I$ is a Veronese type ideal.
\end{proof}

Following Bandari-Herzog \cite{BH}, the set of monomial prime ideals of $R=K[x_1,...,x_n]$ denote by $\mathcal{P}(R)$.
Let $\frak{p}\in\mathcal{P}(R)$ be a monomial prime ideal. Then $\frak{p}=\frak{p}_A$ for some subset $A\subseteq [n]$, where $\frak{p}_A=\lbrace x_i: i\notin A\rbrace$.
From now on, throughout the article, we will denote $a_i=\max\lbrace\deg_{x_i}(u) : u\in G(I)\rbrace$ for $i= 1,...,n$.

\begin{Proposition}\label{P}
Let $I\subset R=K[x_1,x_2,x_3,x_4]$ be a squarefree monomial ideal of degree $d$. Then $I$ is matroidal if and only if $I$ and $I[i]$ have a linear resolution for $i=1,2,3,4$.
\end{Proposition}

\begin{proof}
If $I$ is a matroidal ideal, then by \cite[Corollary 1.2]{BH} $I[i]$ is matroidal for $i=1,2,3,4$ and so they have a linear resolution.
Conversely, we can assume that $\gcd(I)=1$ and so $\gcd(I[i])=1$ for $i=1,2,3,4$. Let $d\geq 3$. Since $\gcd(I[i])=1$, then $I[i]$ is a squarefree monomial ideal in polynomial ring with three variables. Therefore,  by the proof of Lemma \ref{L}, $I[i]$ is a squarefree Veronese type ideal for $i=1,2,3,4$.
Thus, in this case by using \cite[Proposition 2.11]{BH} $I$ is a squarefree Veronese type ideal.
Now, suppose that $d=2$. Since $I$ and $I[i]$ have a linear resolution for $i=1,2,3,4$ hence, for all $\frak{p}\in\mathcal{P}(R)$, $I(\frak{p})$ has a linear resolution. Therefore by using \cite[Corollary 1.2]{BH} $I$ is matroidal. This completes the proof.
\end{proof}

The following example shows that the squarefree condition over $I$ in Proposition \ref{P} is essential.
\begin{Example}
Let $I=(x_1^3,x_1^2x_2,x_1^2x_3,x_2x_3x_4,x_1x_2x_3,x_1x_3x_4,x_1^2x_4)$ be an ideal of $R=k[x_1,x_2,x_3,x_4]$. Then $I$ and $I[i]$ have a linear resolution for $i=1,2,3,4$, while $I$ is not polymatroidal.
\end{Example}
The following example shows that Proposition \ref{P} can not be extend for the polynomial ring with five variables.
\begin{Example}
Let $I=(x_1x_3x_5,x_1x_2x_3,x_1x_2x_4,x_2x_3x_4,x_3x_4x_5,x_2x_4x_5)$ be an ideal of $R=k[x_1,x_2,x_3,x_4,x_5]$. Then $\reg(I)=3$, $I[i]$ is a squarefree monomial ideal of degree $2$ and $\reg(I[i])=2$ for $i=1,2,3,4,5$. Therefore $I$ and $I[i]$ have a linear resolution for $i=1,2,3,4,5$, while $I$ is not matroidal.
\end{Example}

We recall the following definition and result of Herzog-Vladoiu \cite[Theorems 1.1 ~ and~1.3]{HV} which is useful in the sequel of this paper.
\begin{Definition}
A monomial ideal $I$ is {\it intersection type} if $I$ is an intersection of powers of monomial prime ideals.
Also, a monomial ideal $I$ is {\it strong intersection type} if $I$ satisfying the equivalent conditions of the next theorem.
\end{Definition}
\begin{Theorem}\label{P.1}
Let $I$ be a monomial ideal of intersection type with presentation $I=\cap_{{{\frak{p}}}\in\Ass(R/I)}{{\frak{p}}}^{d_{\frak{p}}}$. Then the following statements hold:
\begin{enumerate}
\item[(a)] $d_{\frak{p}}\leq\reg(I({ {\frak{p}}}))$ for all ${\frak{p}}\in\Ass(R/I)$.
\item[(b)] $d_{\frak{p}}=\reg(I({\frak{p}}))$ if and only if $I({\frak{p}})$ has a linear resolution.
\end{enumerate}
In particular, $I=\cap_{{{\frak{p}}}\in\Ass(R/I)}{{\frak{p}}}^{\reg(I({{\frak{p}}}))}$ if and only if $I({\frak{p}})$ has a linear resolution for all ${\frak{p}} \in\Ass(R/I)$.
\end{Theorem}

We need the following proposition for the next results, the proof is similar to \cite[Proposition 2.11]{BH}.
 \begin{Proposition}\label{P.2}
  Let $I$ be a monomial ideal with $d$-linear resolution, and assume that $I({\frak{p}}_{\lbrace i\rbrace})={\frak{p}}_{\lbrace i\rbrace} ^{d-a_i}$ for $i=1,..., n$. Then $I=I_{(d; a_1,..., a_n)}$.
  \end{Proposition}
\begin{proof}
 Since $I$ is generated in degree $d$, we have that $I\subseteq  I_{(d; a_1,..., a_n)}$. Now we want to show
that $I_{(d; a_1,..., a_n)} \subseteq I$. Consider the exact sequence
\[0\longrightarrow I:{\frak{m}}^{\infty}/I
\longrightarrow R/I \longrightarrow R/I:{\frak{m}}^{\infty}\longrightarrow 0.\]
Since $I$ has a $d$-linear resolution, \cite[Corollary 20.19]{E} implies that $(I:{\frak{m}}^{\infty})_d = I_d$, and hence $(I:{\frak{m}}^{\infty})_{\langle d \rangle}= I_{\langle d \rangle}= I$. Here, for any graded ideal $L$, we denote by $L_{{\langle j \rangle}}$ the ideal generated by the jth graded component of $L$. Therefore it is enough to show that $I_{(d;a_1,...,a_n)} \subseteq (I:{\frak{m}}^{\infty})_{\langle d \rangle}$.
Now let $x_{1}^{b_1} . . .  ~ x_{n}^{b_n}  \in G(I_{(d;a_1,...,a_n)})$; then $b_1 + . . .  +b_n = d$ and $b_i \leq  a_i$ for all $i \in [n]$. We claim that $x_{1}^{b_1} . . . ~ x_{n}^{b_n}  {\frak{m}}^s \subseteq  I$ with $s =\sum_{i=1}^{n} a_i -d$. Let $x_{1}^{c_1} . . .  ~ x_{n}^{c_n} \in G({\frak{m}}^s)$. Then $x_{1}^{b_1+c_1} . . . ~  x_{n}^{b_n+c_n} \in x_{1}^{b_1} . . .  ~ x_{n}^{b_n} {\frak{m}}^s$. If $b_i + c_i < a_i$ for all $i$, then $\sum_{i=1}^{n} a_i = d + s =\sum_{i=1}^{n} b_i +\sum_{i=1}^{n} c_i < \sum_{i=1}^{n} a_i$, a contradiction. Hence, without loss of generality, we may assume that $b_1 +c_1 \geq  a_1$, and show that $x_{1}^{b_1+c_1} . . . ~  x_{n}^{b_n+c_n} \in x_{1}^{a_1} {\frak{p}}_{\lbrace 1 \rbrace} ^{d-a_1}\subseteq I$. Since $b_1 + c_1  \geq a_1$, it is enough to show that $x_{2}^{b_2+c_2} . . . ~ x_{n}^{b_n+c_n} \in {\frak{p}}_{\lbrace 1 \rbrace} ^{d-a_1}$.
If $b_i +c_i > d-a_1$ for all $i = 2, . . . , n$, then $x_{2}^{b_2+c_2} . . . ~ x_{n}^{b_n+c_n} \in {\frak{p}}_{\lbrace 1 \rbrace} ^{d-a_1}$.
So we assume $b_i + c_i > d-a_1$ for $i = 2, . . . , t$ that $t \geq 2$ and $b_i + c_i \leq d-a_1$ for $i = t + 1, . . . , n$. Since $t\geq 2$, we have
\[ \sum_{i=2}^{t}d-a_1 + \sum_{i=t+1}^{n}{b_i+c_i}  \geq d-a_1.\]
 This implies that
\[ x_{2}^{d-a_1} . . . ~  x_{t}^{d-a_1} x_{t+1}^{b_{t+1}+c_{t+1}}  . . . ~  x^{b_n+c_n} \in {\frak{p}}_{\lbrace 1 \rbrace} ^{d-a_1}.\]
Since $x_{2}^{b_2+c_2} . . . ~  x^{b_n+c_n} = w(x_{2}^{d-a_1} . . . ~  x_{t}^{d-a_1 }x_{t+1}^{b_{t+1}+c_{t+1}} . . . ~  x^{b_n+c_n} )$ for some monomial $w$, it follows that $x_{2}^{b_2+c_2} . . . ~ x_{n}^{b_n+c_n} \in {\frak{p}}_{\lbrace 1 \rbrace} ^{d-a_1}$, as desired.

\end{proof}

 \begin{Lemma}\label{L.1}
  Let $I$ be a monomial ideal with $d$-linear resolution, and assume that for all ${\frak{p}}\in\mathcal{P}(R)$ the ideal $I({\frak{p}})$ has a linear resolution. Then $I$ is of strong intersection type.
  \end{Lemma}
 \begin{proof}
 Consider the exact sequence
\[0 \longrightarrow I:{\frak{m}}^{\infty}/I
\longrightarrow R/I \longrightarrow R/I:{\frak{m}}^{\infty}\longrightarrow 0.\]
Since $I:{\frak{m}}^{\infty}/I$ is a finite length module, it follows that $\reg(I:{\frak{m}}^{\infty}/I)=\max \lbrace j :  (I:{\frak{m}}^{\infty}/I)_j \ne 0\rbrace$ (see \cite[Lemma 1.1]{CH}). Let $J$ be arbitrary monomial ideal with $d$-linear resolution. Since $J$ has $d$-linear resolution, then $\reg(R/J)=\reg(J)-1=d-1$ and \cite[Corollary 20.19]{E} implies that $(I:{\frak{m}}^{\infty})_{\geq d}=I$. Thus we have $\widetilde I_{\geq d}=I$, where $\widetilde I=(I:{\frak{m}}^{\infty})$. Since $I({\frak{p}})$ has $a_{\frak{p}}$-linear resolution for all ${\frak{p}} \in \mathcal{P}(R)$ and for some integer $a_{\frak{p}}$, then  $\widetilde I({\frak{p}})_{\geq a_{\frak{p}}}=I({\frak{p}})$ that this equivalent to saying that $I({\frak{p}})=\widetilde {I({\frak{p}})} \cap {\frak{m}}_{{\frak{p}}}^{a_{\frak{p}}}$. Then by \cite[Theorem 1.1]{HV}, $I$ is of intersection type and so by Theorem \ref{P.1}, $I$ is of strong intersection type.
  \end{proof}

The following result extends \cite[Proposition 2.8(c)]{BH}.
\begin{Proposition}\label{P.3}
 Let  $I \subseteq R=K[x_1, . . . ,x_n]$ be a monomial ideal generated in degree $d$ and $\height(I)=n-1$. Then $I$ is a Veronese type ideal if and only if for all ${\frak{p}}\in\mathcal{P}(R)$ the ideal $I({\frak{p}})$ has a linear resolution.
 \end{Proposition}

\begin{proof}
$(\Longrightarrow)$. This is clear by \cite[Theorem 3.2]{HRV}.\\
 $(\Longleftarrow).$ Since, for all ${\frak{p}}\in \mathcal{P}(R)$, the ideal $I({\frak{p}})$ has a linear resolution, then by Theorem \ref{P.1}, $I={\frak{p}}_{1}^{d_1} \cap {\frak{p}}_{2}^{d_2} \cap . . . \cap {\frak{p}}_{r}^{d_{r}} \cap {\frak{m}}^d $, where ${\frak{p}}_{i}=(x_1, x_2, ...\widehat {{x_i}}, . . . ,x_{n}) \in\Ass(R/I)$ and $d_{i}=\reg(I({\frak{p}}_{i}))$ for $i=1,2,...,r$. Since $I({\frak{p}}_{i})=I( {\frak{p}}_{\lbrace i \rbrace})=I:x_{i}^{a_{i}}={\frak{p}}_{i}^{d-a_{i}} \cap {\frak{m} }^{d-a_{i}}={\frak{p}}_{i}^{d-a_{i}}$, by using Proposition \ref{P.2}, $I$ is a Veronese type ideal of the form $ I_{(d; a_1, .... ,a_i, ... , a_n)}$.
\end{proof}

 By Proposition \ref{P.3}, we provide a simplified proof for \cite[Proposition 2.4]{BH}.

\begin{Proposition}\label{P.8}
Let $I\subseteq R=K[x_1, . . . ,x_n]$ be a monomial ideal generated in degree $d$, and suppose that $I$ contains at least $n-1$ pure powers of the variables, say $x_{1}^{d_1},..., x_{n-1}^{d}$. Then $I$ is a Veronese type ideal if and only if for all ${\frak{p}}\in\mathcal{P}(R)$ the ideal $I({\frak{p}})$ has a linear resolution.
\end{Proposition}

\begin{proof}
 Since $x_{1}^{d}, . . . , x_{n-1}^{d} \in I$ we have $\height(I)\geq n-1$. Thus if $\height(I)=n$, then $I={\frak{m}}^d$ and if $\height(I)=n-1$, then by Proposition \ref{P.3}, $I$ is a Veronese type of the form $ I_{(d; d, .... ,d, a_n)} $. This completes the proof.
\end{proof}

\begin{Proposition}\label{P.6}
 Let $I \subseteq R= K[x_1,x_2,x_3,x_4]$ be an unmixed monomial ideal of degree $d$, such that $\height(I)=2$. Then $I({ {\frak{p}}})$ has a linear resolution for all ${\frak{p}}\in\mathcal{P}(R)$ if and only if $I$ is a transversal polymatroidal ideal or a squarefree Veronese type ideal.
 \end{Proposition}

\begin{proof}
$(\Longleftarrow)$. If $I$ is a polymatroidal ideal, then for all ${\frak{p}}\in \mathcal{P}(R)$ the ideal $I({\frak{p}})$ has a linear resolution (see \cite[Theorem 3.2]{HRV}).

$(\Longrightarrow)$.
Let for all ${\frak{p}}\in \mathcal{P}(R)$ the ideal $I({\frak{p}})$ has a linear resolution, by Theorem \ref{P.1} we have that $I={\frak{p}}_{1}^{d_1} \cap {\frak{p}}_{2}^{d_2} \cap . . . \cap {\frak{p}}_{r}^{d_r}$ such that $\height({\frak{p}}_i)=2$ and $d_i=\reg(I({\frak{p}}_i))=\deg(I({\frak{p}}_i))$. Therefore by our hypothesis we have $I=(x_2,x_3)^{d_1} \cap (x_2,x_4)^{d_2} \cap (x_3,x_4)^{d_3} \cap (x_1,x_3)^{d_4} \cap (x_1,x_4)^{d_5} \cap (x_1,x_2)^{d_6} $, with this presentation we have the following four cases. \\
{\bf Case}(1): If exactly one of the exponent equal to zero, then, without loss of generality, we may assume that $d_6=0$. Thus $I({\frak{p}}_{\lbrace 3 \rbrace})=(x_2,x_4)^{d_{2}} \cap (x_1,x_4)^{d_{5}}$. Since $I({\frak{p}}_{\lbrace3\rbrace})$ is generated in single degree then we have $G((x_2,x_4)) \cap G((x_1,x_4))= \emptyset$, this is a contradiction.\\
{\bf Case}(2): If exactly two of the exponents equal to zero, then, without loss of generality we may assume that $d_5=d_6=0$. Hence $I({\frak{p}}_{\lbrace 2 \rbrace})=(x_3,x_4)^{d_{3}} \cap (x_1,x_3)^{d_{4}}$. Since $I({\frak{p}}_{\lbrace 2\rbrace})$ is generated in a single degree then we have $G((x_3,x_4)) \cap G((x_1,x_3))= \emptyset$, this is also a contradiction.\\
{\bf Case}(3): If the number of exponents that can be zero bigger than or equal to $3$, then $I={\frak{p}}_{1}^{a_1} \cap {\frak{p}}_{2}^{a_2} \cap {\frak{p}}_{3}^{a_3}$, and so by \cite[Proposition 2.8]{BH}, $I$ is a transversal polymatroidal ideal in this case. \\
{\bf Case}(4): If all exponents are non-zero, then $I({\frak{p}}_{\lbrace1\rbrace})=(x_2,x_3)^{d_{1}} \cap  (x_2,x_4)^{d_{2}} \cap (x_3,x_4)^{d_{3}}$. Therefore by \cite[Proposition 2.8]{BH} we have ${d_{1}}={d_{2}}={d_{3}}=1$ and similarly we have $d_1= . . . =d_6=1$. Therefore $I=(x_2,x_3) \cap (x_2,x_4)\cap (x_3,x_4) \cap (x_1,x_3) \cap (x_1,x_4) \cap (x_1,x_2) = (x_{1}x_{2}x_{3}, x_{1}x_{2}x_{4}, x_{1}x_{3}x_{4}, x_{2}x_{3}x_{4})$, which is a squarefree Veronese type ideal. This completes the proof.
\end{proof}

\begin{Proposition}\label{P.7}
 Let $I \subseteq R= K[x_1,x_2,x_3,x_4]$ be a monomial ideal of degree $d$ with no embedded prime ideals. Then $I({ {\frak{p}}})$ has a linear resolution for all ${\frak{p}} \in \mathcal{P}(R)$ if and only if $I$ is a transversal polymatroidal ideal or a squarefree Veronese type ideal either a matroidal ideal of degree $2$.
\end{Proposition}

\begin{proof}
$(\Longleftarrow)$ immediately follows by \cite[Theorem 3.2]{HRV}.\\
$(\Longrightarrow)$.
If $\frak{m}\in\Ass(R/I)$, then $I=\frak{m}^d$ and the result follows in this case. Let $\frak{m}\notin\Ass(R/I)$. If $I$ is an unmixed ideal, then by Propositions \ref{P.3} and \ref{P.6} the result follows. Now, we assume that $I$ is not an unmixed ideal and so we have $\frak{p}_1,\frak{p}_2 \in\Ass(R/I)$ such that $\height(\frak{p}_1)=3$ and $\height(\frak{p}_2)=2$. Hence, without loss of generality we may assume that $\frak{p}_1=(x_2,x_3,x_4)$ and $\frak{p}_2=(x_1,x_3)$. Since $I$ has no embedded prime ideals, then $(x_2,x_3), (x_2,x_4),(x_3,x_4), (x_1,x_2,x_3), (x_1,x_3,x_4)$ are not in $\Ass(R/I)$. We claim that $(x_1,x_4),(x_1,x_2)$ are not in $\Ass(R/I)$. Indeed, let $(x_1,x_4)\in\Ass(R/I)$. Then $I(\frak{p}_{\lbrace2\rbrace})=(x_1,x_3)^{a_1} \cap (x_1,x_4)^{a_2}$, for some integer $a_1$ and $a_2$. Since $I(\frak{p}_{\lbrace 2\rbrace})$ is generated in a single degree then we have $G((x_1,x_3)) \cap G((x_1,x_4))= \emptyset$, this is a contradiction. Similarly, $(x_1,x_2) \notin\Ass(R/I)$. Therefore $I=(x_2,x_3,x_4)^{d_1} \cap (x_1,x_3)^{d_2} \cap (x_1,x_2,x_4)^{d_3}$ and so by \cite[Proposition 2.8]{BH}  $I$ is a transversal polymatroidal ideal or a matroidal ideal of degree $2$. This completes the proof.
\end{proof}

Let $I\subseteq R=K[x_1,..., x_n]$ be a monomial ideal generated in degree $d$. Bandari-Herzog \cite[Proposition 2.8]{BH}, proved that if $I$ contains at least $n-1$ pure powers of the variables, say $x_{1}^d , . . .  , x_{n-1}^d$, then for all $\frak{p}\in\mathcal{P}(R)$ the ideal $I({{\frak{p}}})$ has a linear resolution if and only if $I$ is a Veronese type ideal of the form $I=I_{(d; d, d,..., d, k)}$ for some $k$. The following example shows that if $I$ contains $n-2$ pure powers of the variables, say $x_{1}^d ,... , x_{n-2}^d$ and for all $\frak{p}\in \mathcal{P}(R)$ the ideal $I({ {\frak{p}}})$ has a linear resolution , then $I$ is not a Veronese type ideal. In the following proposition, with this condition, we prove that $I$ is a polymatroidal ideal which is a special case of Bandari-Herzog's Conjecture \cite[Conjecture 2.9]{BH}.

\begin{Example}
Let $I=(x_1,x_2)^2 \cap (x_1,x_2,x_3)^3 \cap (x_1,x_2,x_4)^3 \cap (x_1,x_2,x_3,x_4)^5 \subseteq R= K[x_1,x_2,x_3,x_4]$ be a monomial ideal of degree $5$. The ideal $I$ contain $2$ pure power of the variables, $x_1^{5}$ and $x_2^{5}$. $a_3= \max \lbrace\deg_{x_3}(w) : w \in G(I) \rbrace =2$, but $x_{3}^{2}x_{4}^{2} x_1 \notin I$. Hence $I$ is not a Veronese type ideal but by the following proposition $I$ is polymatroidal.
\end{Example}

\begin{Theorem}\label{P.8}
Let $I \subseteq R=K[x_1,..., x_n]$ be a monomial ideal generated in degree $d$, and suppose that $I$ contains at least $n-2$ pure powers of the variables, say $x_{1}^d , . . .  , x_{n-2}^d$. Then $I$ is polymatroidal ideal if and only if $I({{\frak{p}}})$ has a linear resolution for all ${\frak{p}} \in \mathcal{P}(R)$.
\end{Theorem}

\begin{proof}
$(\Longrightarrow)$ immediately follows by \cite[Theorem 3.2]{HRV}.\\
$(\Longleftarrow)$. Since for all ${\frak{p}}\in\mathcal{P}(R)$ the ideal $I({\frak{p}})$ has a linear resolution, by Theorem \ref{P.1} $I={\frak{p}}_{1}^{d_1} \cap {\frak{p}}_{2}^{d_2} \cap . . . \cap {\frak{p}}_{r}^{d_r}$, where $d_i=\reg(I({\frak{p}}_i))=\deg(I({\frak{p}}_i))$. Since $x_{1}^d , . . .  , x_{n-2}^d \in I$, then $\height(I) \geq n-2$. If $\height(I)\geq n-1$, then by Proposition \ref{P.3}, $I$ is a Veronese type ideal. Hence, we assume that $\height(I)=n-2$, and we have the following presentation:
\[I=(x_1, . . . x_{n-2})^{d_1} \cap (x_1, . . . , x_{n-2},x_{n-1})^{d_2} \cap (x_1, . . . , x_{n-2},x_{n})^{d_3} \cap {\frak{m}}^{d_4} .\]
 Since for all ${\frak{p}} \in \mathcal{P}(R)$ the ideal $I({ {\frak{p}}})$ has a linear resolution, then $I({\frak{p}_{\lbrace n-1 \rbrace}})=(x_1, . . . x_{n-2})^{d_1} \cap  (x_1, . . . , x_{n-2},x_{n})^{d_3}$. If $(x_1, . . . , x_{n-2},x_{n})\in\Ass(R/I)$, then $ d_3=\deg(I({\frak{p}_{\lbrace n-1 \rbrace }}))=d-a_{n-1}$ otherwise $d_{3}=0$. Similarly if $(x_1, . . . , x_{n-2},x_{n-1}) \in\Ass(R/I)$, then $d_2=\deg(I({\frak{p}_{\lbrace n \rbrace }}))=d-a_{n}$  otherwise $d_{2}=0$ and also if ${\frak{m}} \in\Ass(R/I)$, then $d_{4}=d$ otherwise $d_{4}=0$.
Since $\height(I)=n-2$, then $d_1 \ne 0$ and $d_1 < d_2,d_3$.
Note that $u= x_{1}^{r_1} . . . x_{n}^{r_{n}} \in G(I)$ if and only if
$\deg(u)=d$ and we have the following inequalities:
\begin{align*}
(1) \quad \quad \quad \quad \quad \quad ~ r_1 + . . .  +r_{n-2} \geq d_1 \\
(2) \quad \quad \quad r_1 + . . .  +r_{n-2}+ r_{n-1} \geq d_2 \\
(3)  \quad \quad \quad \quad r_1 + . . .  +r_{n-2}+ r_{n} \geq d_3 \\
(4) \quad r_1 + . . .  +r_{n-2}+ r_{n-1}+ r_{n} =d.
\end{align*}
Let $u= x_{1}^{r_1} . . . x_{n}^{r_{n}}, v= x_{1}^{s_1} . . . x_{n}^{s_{n}}\in G(I)$  and $\deg_{x_i}(u) > \deg_{x_i}(v)$. We want to show that there exists a variable $x_j$ such that $\deg_{x_j}(u) <\deg_{x_j}(v)$ and $x_{j}(u/x_{i}) \in I$. There are two cases. \\
{\bf Case 1:}
Suppose that $r_{i} > s_{i}$ for some $i= 1, . . . ,n-2$. Without loss of generality, we can assume that $r_{1} > s_{1}$. If there exists $r_{j} < s_{j}$ for some $j= 2,...,n-2$, then $x_{j}(u/x_1) \in I$, since none of the left-hand sides of the inequalities above
change, and so we have the result. Thus we assume that $r_{j} \geq s_{j}$ for all $ 2 \leq j \leq n-2$. Since $u$ and $v$ have the same degree, there exists some $r_j < s_j$  with $n-1 \leq j \leq n$. Without loss of generality, we assume that $r_{n-1} < s_{n-1}$. The monomial $x_{n-1}(u/x_1)$ satisfies in all inequalities except $(3)$. If $r_1 + . . .  +r_{n-2}+ r_{n} > d_3$, then all inequalities hold. If $r_1 + . . .  +r_{n-2}+ r_{n} = d_3=d-a_{n-1}$, then $r_{n-1}=a_{n-1}$ and this is a contradiction.\\
{\bf Case 2:} Suppose that $r_i > s_i $ for some $i=n-1,n$. Without loss of generality, we assume that $r_{n-1} > s_{n-1}$. If there exists $r_{j} < s_{j}$ for some $j= 1 , . . . , n-2$, then $x_{j}(u/x_{n-1}) \in I$ since none of the left-hand sides of the inequalities above change, and we have the result. Suppose that $r_{j} \geq s_{j}$ for all $j= 1, . . . , n-2$. Thus we can assume that $r_{n} < s_{n}$. The monomial $x_{n}(u/x_{n-1})$ satisfies in all inequalities except $(2)$. If $r_1 + . . .  +r_{n-2}+ r_{n-1} > d_2$, then $x_{n}(u/x_{n-1}) \in I$, and we have the result. If $r_1 + . . .  +r_{n-2}+ r_{n-1} =d_2= d-{a_n}$, then $r_{n}=a_{n}$ and this is a contradiction. This completes the proof.
\end{proof}

\begin{Proposition}\label{P.7}
 Let $I \subseteq R= K[x_1,x_2,x_3,x_4]$ be a monomial ideal of degree $d$. Then $I$ is polymatroidal ideal if and only if $I({ {\frak{p}}})$ has a linear resolution for all ${\frak{p}} \in \mathcal{P}(R)$.
 \end{Proposition}
\begin{proof}
$(\Longrightarrow)$ immediately follows by \cite[Theorem 3.2]{HRV}.\\
$(\Longleftarrow)$. Since for all ${\frak{p}} \in \mathcal{P}(R)$ the ideal $I({\frak{p}})$ has a linear resolution, by Theorem \ref{P.1} $I={\frak{p}}_{1}^{d_1} \cap {\frak{p}}_{2}^{d_2} \cap ... \cap {\frak{p}}_{r}^{d_r}$, such that $d_i=\reg(I({\frak{p}}_i))=\deg(I({\frak{p}}_i))$. We can assume that $\height(I)\geq 2$. If $\height(I)\geq 3$, then by Proposition \ref{P.3}, $I$ is a Veronese type ideal. Thus we assume that $\height(I)=2$, and we have the following presentation:
\begin{align*}
 I&=(x_2,x_3)^{d_1} \cap (x_2,x_4)^{d_2} \cap (x_3,x_4)^{d_3} \cap (x_1,x_3)^{d_4} \cap (x_1,x_4)^{d_5} \cap (x_1,x_2)^{d_6} \\
 &  \quad \cap (x_2,x_3,x_4)^{d_7} \cap (x_1,x_3,x_4)^{d_8} \cap (x_1,x_2,x_4)^{d_9} \cap (x_1,x_2,x_3)^{d_{10}} \\
 & \quad \cap {\frak{m}}^{d_{11}}.
\end{align*}
 Since for all ${\frak{p}} \in \mathcal{P}(R)$ the ideal $I({{\frak{p}}})$ has a linear resolution, then
\[I({\frak{p}_{\lbrace 1 \rbrace}})=(x_2,x_3)^{d_1} \cap (x_2,x_4)^{d_2} \cap (x_3,x_4)^{d_3} \cap  (x_2,x_3,x_4)^{d_7}. \]
If $(x_2,x_3,x_4)\in\Ass(R/I)$, then $ d_7=\deg(I({\frak{p}_{\lbrace 1 \rbrace }}))=d-a_{1}$ otherwise $d_7=0$, similarly if $(x_1,x_3,x_4) \in\Ass(R/I)$, then $ d_8=\deg(I({\frak{p}_{\lbrace 2 \rbrace }}))=d-a_{2}$ otherwise $d_8=0$, if $(x_1,x_2,x_4) \in\Ass(R/I)$, then $ d_9=\deg(I({\frak{p}_{\lbrace 3 \rbrace }}))=d-a_{3}$ otherwise $d_9=0$, if $(x_1,x_2,x_3) \in\Ass(S/I)$, then $ d_{10}=\deg(I({\frak{p}_{\lbrace 4 \rbrace }}))=d-a_{4}$, otherwise $d_{10}=0$ and if ${\frak{m}} \in\Ass(R/I)$, then $d_{11}=d$ otherwise $d_{11}=0$.
Note that $u= x_{1}^{r_1}x_{2}^{r_2}x_{3}^{r_3}x_{4}^{r_4}  \in G(I)$ if and only if
$\deg(u)=d$ and we have the following inequalities:

\[\begin{array}{*{20}{c}}
{(1) \quad ~~~ r_2 +r_3  \geq d_1}&{(5) \quad ~~~ r_1 +r_3  \geq d_4}&{(8) \quad ~~~ r_1 +r_2  \geq d_6}\\
{(2) \quad ~~~ r_2 +r_4  \geq d_2}&{(6) \quad ~~~ r_1 +r_4  \geq d_5}&{(9)r_1 +r_2 +r_4 \geq d_9}\\
{(3) \quad ~~~ r_3 +r_4  \geq d_3}&{(7)r_1 +r_3 +r_4 \geq d_8}&{(10)r_1 +r_2 +r_3 \geq d_{10}}.\\
{(4)r_2 +r_3+r_4  \geq d_7}&{}&{}
\end{array}\]
 \\
Let $u= x_{1}^{r_1}x_{2}^{r_2}x_{3}^{r_3}x_{4}^{r_4}, v= x_{1}^{s_1}x_{2}^{s_2}x_{3}^{s_3}x_{4}^{s_4} \in G(I) $ with $\deg_{x_i}(u)>\deg_{x_i}(v)$. We want to show that there exists variable $x_j$ such that $\deg_{x_j}(u) <\deg_{x_j}(v)$ and $x_{j}(u/x_{i}) \in I$. Without loss of generality, we assume that $r_{1} > s_{1}$. There are three main cases.\\
\textbf{ Case 1:} If $r_{2} < s_{2}$, $r_{3} \geq s_{3}$ and $r_{4} \geq s_{4}$, then $x_2(u/x_1)$ satisfies in all inequalities except $(5)$, $(6)$ and $(7)$. If $r_1+r_3 >d_4$, $r_1+r_4 >d_5$ and $r_1+r_3+r_4 > d_8$, then $x_2(u/x_1) \in I$ since none of the left-hand sides of the inequalities above change. Indeed, if $ r_1+r_3 \not  > d_4$, then
 \[r_1+r_3=d_4 \leq s_1+s_3\]
 that this is a contradiction since $r_1 > s_1$ and $r_3 \geq s_3$. Similarly, $r_1+r_4 > d_5$ and also $r_1+r_3+r_4 > d_8$. Therefore $x_2(u/x_1) \in I$.
\\
\textbf{Case 2:} Suppose $r_{2} < s_{2}$, $r_{3} < s_{3}$ and $r_{4} \geq s_{4}$. Without loss of generality, we assume that $x_3(u/x_1) \notin I$, we claim that $x_2(u/x_1) \in I$. If $r_4=a_4$ then $w=u:x_{4}^{a_4}=x_{1}^{r_1}x_{2}^{r_2}x_{3}^{r_3} \in G(I:x_{4}^{a_4})$. Since $v:x_{4}^{a_4}= x_{1}^{s_1}x_{2}^{s_2}x_{3}^{s_3} \in I:x_{4}^{a_4}$, there exist  $z=x_{1}^{t_1}x_{2}^{t_2}x_{3}^{t_3} \in G(I:x_{4}^{a_4}) $ such that $z \vert v:x_{4}^{a_4}$. Since $I:x_{4}^{a_4}$  is polymatroidal, $r_1> s_1 \geq t_1$, $\deg(z)=\deg(w)$ and $x_3(u/x_1) \notin I$, it follows $x_2(w/x_1) \in I:x_{4}^{a_4}$ and so $x_2(u/x_1) \in I$. Now, we assume that $r_4 < a_4$ and we will prove the claim.

Since $x_3(u/x_1)$ satisfies in all inequalities except $(6)$, $(8)$ and $(9)$, it follows $r_1+r_4=d_5$ or $r_1+r_2=d_6$ or $r_1+r_2+r_4=d-a_3$. If $r_1+r_4 \not > d_5$, then \[r_1+r_4=d_5 \leq s_1+s_4\]
 and this is a contradiction because by our hypothesis $r_1 > s_1$ and $r_4 \geq s_4$. If $ r_1+r_2+r_4 =d-a_3$, then $r_3=a_3 \geq s_3$ and this is a contradiction with $r_3 < s_3$. Hence if $x_3(u/x_1)\notin I$, then we have $r_1+r_2=d_6$. In order to prove the claim, we assume by contrary $x_2(u/x_1) \notin I$. The monomial $x_2(u/x_1)$ satisfies in all inequalities except $(5)$, $(6)$ and $(7)$. If $r_1 +r_3  > d_4$, $r_1 +r_4  > d_5$ and $r_1 +r_3 +r_4 > d-a_2$, then $x_2(u/x_1) \in I$ because all inequalities holds. Since $r_1 > s_1$ and $r_4 \geq s_4$ then we have $r_1+r_4 > s_1+s_4 \geq d_5$ and so we have $(6)$. If $r_1 +r_3 +r_4 = d-a_2$, then $r_2=a_2 \geq s_2$ and this is a contradiction and so we have $(7)$. Hence if $x_2(u/x_1) \notin I$, then we have $r_1+r_3 = d_4$.
 Since $u: x_{4}^{a_4} = x_{1}^{r_1} x_{2}^{r_2} x_{3}^{r_3} \in I:x_{4}^{a_4}$, there exists $z=x_{1}^{t_1}x_{2}^{t_2}x_{3}^{t_3} \in G(I:x_{4}^{a_4}) $ such that $z \vert u:x_{4}^{a_4}$ and $t_1+t_2 \geq d_6$, $t_1+t_3 \geq d_4 $ and $t_1+t_2+t_3=d-a_4$. Therefore we have
\[ r_1+d-a_4 < r_1+r_1+r_2+r_3=d_6+d_4\leq t_1+t_1+t_2+t_3=t_1+d-a_4\] and so $r_1 < t_1$ which is a contradiction. Hence $x_2(u/x_1)\in I$.\\
\textbf{Case 3:} $r_i < s_i$ for $i=2,3,4$.  Without loss of generality, we assume that $x_4(u/x_1) \notin I $. Then we have $r_1+r_2 = d_6$ or $r_1+r_3=d_4$. Also, we assume that $x_3(u/x_1) \notin I $. Therefore we have $r_1+r_2 = d_6$ or $r_1+r_4=d_5$. Now, we claim that $x_2(u/x_1) \in I$.  In order to prove the claim, we assume by contrary that $x_2(u/x_1) \notin I$. The monomial $x_2(u/x_1)$ satisfies in all inequalities except $(5)$, $(6)$ and $(7)$.
 If $r_1 +r_3  > d_4$, $r_1 +r_4  > d_5$ and $r_1 +r_3 +r_4 > d-a_2$, then $x_2(u/x_1) \in I$ because all inequalities hold. If $x_2(u/x_1) \notin I$, then we have $r_1+r_3 = d_4$  or $r_1+r_4=d_5$.
 Thus we can consider the following cases:
 \begin{enumerate}
 \item[(A)] $r_1+r_3=d_4$ and $r_1+r_4=d_5$,
 \item[(B)] $r_1+r_3=d_4$ and $r_1+r_2=d_6$,
 \item[(C)] $r_1+r_4=d_4$ and $r_1+r_2=d_6$.
 \end{enumerate}
\noindent
{\bf Case}(A): Since $r_2 < s_2 \leq a_2$, then $r_1+r_3+r_4 > d-a_2$ and there exist $z=x_{1}^{t_1}x_{3}^{t_3} x_{4}^{t_4} \in G(I:x_{2}^{a_2}) $ such that $z \vert u:x_{2}^{a_2}$ and $t_1+t_3 \geq d_4$, $t_1+t_4 \geq d_5 $ and $t_1+t_3+t_4=d-a_2$. Therefore we have
\[ r_1+d-a_2 < r_1+r_1+r_3+r_4=d_6+d_5\leq t_1+t_1+t_3+t_4=t_1+d-a_2\] and so $r_1 < t_1$ which is a contradiction.\\
{\bf Case}(B): By using a similar proof as in Case 2 we have $x_2(u/x_1)\in I$.\\
{\bf Case}(C): This follows by a similar argument of {\bf Case}(2) and {\bf Case}(A). This completes the proof.
\end{proof}

 \begin{Theorem}\label{P.9}
Let $I \subseteq R=K[x_1, . . .  , x_n]$ be a monomial ideal generated in degree $d$, and suppose that $I$ contains $n-3$ pure powers of the variables, say $x_{1}^d , . . .  , x_{n-3}^d$. Then $I$ is a polymatroidal ideal if and only if $I({ {\frak{p}}})$ has a linear resolution for all ${\frak{p}} \in \mathcal{P}(R)$.
 \end{Theorem}

\begin{proof}
$(\Longrightarrow)$ immediately follows by \cite[Theorem 3.2]{HRV}.\\
$(\Longleftarrow)$. Since for all ${\frak{p}} \in \mathcal{P}(R)$ the ideal $I({\frak{p}})$ has a linear resolution, by Theorem \ref{P.1} we have $I={\frak{p}}_{1}^{d_1} \cap {\frak{p}}_{2}^{d_2} \cap . . . \cap {\frak{p}}_{r}^{d_r}$, such that $d_i=\reg(I({\frak{p}}_i))=\deg(I({\frak{p}}_i))$. Since $x_{1}^d , . . .  , x_{n-3}^d \in I$, then $\height(I)\geq n-3$, and we have the following presentation:
\begin{align*}
 I&=(x_1, . . . ,x_{n-3})^{d_1} \cap (x_1, . . . , x_{n-3},x_{n-2})^{d_2} \cap (x_1, . . . , x_{n-3},x_{n-1})^{d_3} \cap (x_1, . . . , x_{n-3},x_{n})^{d_4} \\
 &  \quad  \cap (x_1, . . . , x_{n-3},x_{n-2},x_{n-1})^{d_5} \cap (x_1, . . . , x_{n-3},x_{n-2},x_{n})^{d_6} \cap (x_1, . . . , x_{n-3},x_{n-1},x_{n})^{d_7}\\
  &  \quad  \cap {\frak{m}}^{d_8} .
  \end{align*}
Since for all ${\frak{p}} \in \mathcal{P}(R)$ the ideal $I({ {\frak{p}}})$ has a linear resolution, then
\[I({\frak{p}_{\lbrace n-2 \rbrace}})=(x_1, . . ., x_{n-3})^{d_1} \cap  (x_1, . . . , x_{n-3},x_{n-1})^{d_3} \cap  (x_1, . . . , x_{n-3},x_{n})^{d_4} \cap  (x_1, . . . , x_{n-3},x_{n-1},x_{n})^{d_7}.\]
Therefore if $(x_1, . . . , x_{n-3},x_{n-1},x_{n}) \in\Ass(R/I)$, then $ d_7=\deg(I({\frak{p}_{\lbrace n-2 \rbrace }}))=d-a_{n-2}$ otherwise $d_{7}=0$,  similarly if $(x_1, . . . ,x_{n-3}, x_{n-2},x_{n}) \in\Ass(R/I)$, then $d_6=\deg(I({\frak{p}_{\lbrace n-1 \rbrace }}))=d-a_{n-1}$  otherwise $d_{6}=0$, if $(x_1, . . . ,x_{n-3}, x_{n-2},x_{n-1}) \in\Ass(R/I)$, then $d_5=\deg(I({\frak{p}_{\lbrace n \rbrace }}))=d-a_{n}$  otherwise $d_{5}=0$ and if ${\frak{m}} \in\Ass(R/I)$, then $d_{8}=d$ , otherwise $d_{8}=0$.
Note that $u= x_{1}^{r_1} . . . x_{n}^{r_{n}} \in G(I)$ if and only if
$\deg(u) = d$ and all of the following inequalities hold:
\[\begin{array}{*{20}{c}}
{(1) \quad \quad \quad \quad  r_1 + . . .  +r_{n-3} \geq d_1}&{(5) \quad \quad    r_1 + . . .  +r_{n-3}+r_{n-2}+r_{n-1} \geq d_5} \\
{(2) \quad  r_1 + . . .  +r_{n-3}+r_{n-2} \geq d_2}&{(6) \quad \quad \quad   r_1 + . . .  +r_{n-3}+r_{n-2}+r_{n} \geq d_6}\\
{(3) \quad   r_1 + . . .  +r_{n-3}+r_{n-1} \geq d_3}&{(7) \quad \quad \quad  r_1 + . . .  +r_{n-3}+r_{n-1}+r_{n} \geq d_7}\\
{(4) \quad \quad  r_1 + . . .  +r_{n-3}+r_{n} \geq d_4}&{(8) r_1 + . . .  +r_{n-3}+r_{n-2}+r_{n-1}+r_{n}=d. }
\end{array}\]
Let $u= x_{1}^{r_1} . . . x_{n}^{r_{n}}, v= x_{1}^{s_1} . . . x_{n}^{s_{n}} \in G(I) $ with $\deg_{x_i}(u) >\deg_{x_i}(v)$. We want to show that there exists a variable $x_j$ such that $\deg_{x_j}(u) <\deg_{x_j}(v)$ and $x_{j}(u/x_{i}) \in I$. There are two main cases.\\
{\bf First case}:  Suppose that $r_{i} > s_{i}$ for some $i= 1, . . . ,n-3$. Without loss of generality, we assume that $r_{1} > s_{1}$. If there exists $r_{j} < s_{j}$ for some $j= 2 , . . . , n-3$, then $x_{j}(u/x_1) \in I$ since none of the left-hand sides of the inequalities above
 change, and so we have the result. Hence we assume that $r_{j} \geq s_{j}$ for all $ 2 \leq j \leq n-3$. Since $u$ and $v$ have the same degree, there exists some $r_j < s_j$  with $n-2 \leq j \leq n$. So we have the following three cases.\\
\textbf{ Case 1:} If $r_{n-2} < s_{n-2}$, $r_{n-1} \geq s_{n-1}$ and $r_{n} \geq s_{n}$, then $x_{n-2}(u/x_1)$ satisfies in all inequalities except $(3)$, $(4)$ and $(7)$. If $r_1 + . . .  +r_{n-3}+r_{n-1} > d_3$, $ r_1 + . . .  +r_{n-3}+r_{n} > d_4$ and $ r_1 + . . .  +r_{n-3}+r_{n-1}+r_{n} > d_7$, then $x_{n-2}(u/x_1) \in I$  since none of the left-hand sides of the inequalities above change. If $r_1 + . . .  +r_{n-3}+r_{n-1} \not > d_3$, then
 \[r_1 + . . .  +r_{n-3}+r_{n-1} = d_3 \leq s_1 + . . .  +s_{n-3}+s_{n-1} \]
 and this is a contradiction because by our hypothesis $r_1 > s_1$, $r_j \geq s_j$ for $j=2, . . . ,n-3$ and $r_{n-1} \geq s_{n-1}$. Similarly $ r_1 + . . .  +r_{n-3}+r_{n} > d_4$ and $r_1 + . . .  +r_{n-3}+r_{n-1}+r_{n} > d_7$ and so $x_{n-2}(u/x_1) \in I$.
\\
\textbf{Case 2:} Suppose $r_{n-2} < s_{n-2}$, $r_{n-1} < s_{n-1}$ and $r_{n} \geq s_{n}$. Without loss of generality, we assume that $x_{n-1}(u/x_1) \notin I$. Now, we claim that $x_{n-2}(u/x_1) \in I$. If $r_{n}=a_{n}$, then $w=u:x_{n}^{a_n}=x_{1}^{r_1}x_{2}^{r_2}~ . . . ~ x_{n-1}^{r_{n-1}} \in G(I:x_{n}^{a_n})$. Since $v:x_{n}^{a_n}= x_{1}^{s_1}x_{2}^{s_2} ~. . .~ x_{n-1}^{s_{n-1}} \in I:x_{n}^{a_n}$, there exists $z=x_{1}^{t_1}x_{2}^{t_2}~ . . .~ x_{n-1}^{t_{n-1}} \in G(I:x_{n}^{a_n}) $ such that $z \vert v:x_{n}^{a_n}$. By Theorem \ref{P.8}, $I:x_{n}^{a_n}$  is polymatroidal. Since $r_1> s_1 \geq t_1$, $\deg(z)=\deg(w)$ and $x_{n-1}(u/x_1) \notin I$ then we have $x_{n-2}(w/x_1) \in I:x_{n}^{a_n}$ and so $x_{n-2}(u/x_1) \in I$. Thus we assume that $r_{n} < a_n$ and we want to prove the claim.
 Since $x_{n-1}(u/x_1)$ satisfies in all inequalities except $(2)$, $(4)$ and $(6)$, we have $ r_1 + . . .  +r_{n-3}+r_{n-2} = d_2$ or $r_1 + . . .  +r_{n-3}+r_{n} = d_4$ or $ r_1 + . . .  +r_{n-3}+r_{n-2}+r_{n} = d-{a_{n-1}}$. If $r_1 + . . .  +r_{n-2}+r_{n} \not > d_4$, then \[r_1 + . . .  +r_{n-3}+r_{n}=d_4 \leq s_1 + . . .  +s_{n-3}+s_{n}\]
 and this is a contradiction because by our hypothesis $r_1 > s_1$, $r_j \geq s_j$ for $j=2, . . . ,n-3$ and $r_{n} \geq s_{n}$. If $ r_1 + . . .  +r_{n-3}+r_{n-2}+r_{n} = d-{a_{n-1}}$, then $r_{n-1}=a_{n-1} \geq s_{n-1}$ and this is a contradiction with $r_{n-1} < s_{n-1}$. Therefore if $x_{n-1}(u/x_1) \notin I$ we have $r_1 + . . .  +r_{n-3}+r_{n-2} = d_2$. In order to prove the claim, we assume by contrary that $x_{n-2}(u/x_1) \notin I$. The monomial $x_{n-2}(u/x_1)$ satisfies in all inequalities except $(3)$, $(4)$ and $(7)$. If $ r_1 + . . .  +r_{n-3}+r_{n-1} > d_3$, $r_1 + . . .  +r_{n-3}+r_{n} > d_4$ and $ r_1 + . . .  +r_{n-3}+r_{n-1}+r_{n} > d_7$, then $x_{n-2}(u/x_1) \in I$ because all inequalities hold. Since $r_1 > s_1$, $r_j \geq s_j$ for $j=2, . . . ,n-3$ and $r_{n} \geq s_{n}$ then $r_1 + . . .  +r_{n-3}+r_{n} > s_1 + . . .  +s_{n-3}+s_{n} \geq d_4$. If $ r_1 + . . .  +r_{n-3}+r_{n-1}+r_{n} = d-{a_{n-2}}$, then $r_{n-2}=a_{n-2} \geq s_{n-2}$ and this is a contradiction. Hence if $x_{n-2}(u/x_1) \notin I$, we have $r_1 + . . .  +r_{n-3}+r_{n-1} = d_3$.
 Since $u: x_{n}^{a_n} = x_{1}^{r_1} x_{2}^{r_2} ~. . . ~ x_{n-1}^{r_{n-1}} \in I$, there exists $z=x_{1}^{t_1}x_{2}^{t_2}~ . . .~ x_{n-1}^{t_{n-1}} \in G(I:x_{n}^{a_n}) $ such that $z \vert u:x_{n}^{a_n}$ and $t_1 + . . .  +t_{n-3}+t_{n-2} \geq d_2$, $t_1 + . . .  +t_{n-3}+t_{n-1} \geq d_3 $ and $t_1 + . . .  +t_{n-3}+t_{n-2}+t_{n-1} = d-{a_n}$. Therefore we have
 \begin{align*}
 r_1+. . . +r_{n-3}+d-a_n &< r_1+ . . . +r_{n-3}+r_1+ . . . +r_{n-3}+r_{n-2}+r_{n-1}\\
 &=d_2+d_3\\
 &\leq t_1+ . . . +t_{n-3}+t_1+ . . . +t_{n-3}+t_{n-2}+t_{n-1}\\
 &=t_1+. . . +t_{n-3}+d-a_n.
 \end{align*}
 Thus $r_1+. . . +r_{n-3} < t_1+. . . +t_{n-3}$ and this is a contradiction. Thus $x_{n-2}(u/x_1) \in I$.\\
\textbf{Case 3:} Suppose $r_i < s_i$ for $i=n-2,n-1,n$.  Without loss of generality, we assume that $x_n(u/x_1) \notin I $ then we have $r_1 + . . .  +r_{n-3}+r_{n-2} = d_2$ or $ r_1 + . . .  +r_{n-3}+r_{n-1} = d_3$. Also we assume that $x_{n-1}(u/x_1) \notin I $ then we have $r_1 + . . .  +r_{n-3}+r_{n-2} = d_2$ or $r_1 + . . .  +r_{n-3}+r_{n} = d_4$. Now we claim that $x_{n-2}(u/x_1) \in I$.  In order to prove the claim, we assume by contrary that $x_{n-2}(u/x_1) \notin I$. The monomial $x_{n-2}(u/x_1)$ satisfies in all inequalities except $(3)$, $(4)$ and $(7)$.
 If $ r_1 + . . .  +r_{n-3}+r_{n-1} > d_3$, $r_1 + . . .  +r_{n-3}+r_{n} > d_4$, then $x_{n-2}(u/x_1) \in I$ because all inequalities hold. Hence if $x_{n-2}(u/x_1) \notin I$, we have $r_1 + . . .  +r_{n-3}+r_{n-1} = d_3$  or $r_1 + . . .  +r_{n-3}+r_{n} = d_4$.
 Now we consider the following cases:
 \begin{enumerate}
 \item[(A)] $r_1 + . . .  +r_{n-3}+r_{n-1} = d_3$ and $r_1 + . . .  +r_{n-3}+r_{n} = d_4$,
 \item[(B)] $r_1 + . . .  +r_{n-3}+r_{n-1} = d_3$ and $r_1 + . . .  +r_{n-3}+r_{n-2} = d_2$,
 \item[(C)] $r_1 + . . .  +r_{n-3}+r_{n} = d_4$ \quad and $r_1 + . . .  +r_{n-3}+r_{n-2} = d_2$ .
 \end{enumerate}
\noindent
$(A)$: Since $r_{n-2} < s_{n-2} \leq a_{n-2}$, then $ r_1 + . . .  +r_{n-3}+r_{n-1}+r_{n} > d-{a_{n-2}}$ and there exist $z=x_{1}^{t_1}~ . . . ~ x_{n-3}^{t_{n-3}}x_{n-1}^{t_{n-1}} x_{n}^{t_n} \in G(I:x_{n-2}^{a_{n-2}}) $ such that $z \vert u:x_{n-2}^{a_{n2}}$ and $t_1 + . . .  +t_{n-3}+t_{n-1} \geq d_3$, $t_1 + . . .  +t_{n-3}+t_{n} \geq d_4$  and $ t_1 + . . .  +t_{n-3}+t_{n-1}+t_{n} =d-{a_{n-2}}$. Thus we have
 \begin{align*}
 r_1+ . . . +r_{n-3}+d-a_{n-2} &< r_1+ . . . +r_{n-3}+r_1+ . . . +r_{n-3}+r_{n-1}+r_{n}\\
 &=d_3+d_4\\
 &\leq  t_1+ . . . +t_{n-3}+t_1+ . . . +t_{n-3}+t_{n-1}+t_{n}\\
 &= t_1+ . . . +t_{n-3}+d-a_{n-2}.
\end{align*}
 Thus $r_1+. . . +r_{n-3} < t_1+. . . +t_{n-3}$ and this is a contradiction. Hence $x_{n-2}(u/x_1) \in I$.\\
$(B)$: By similar proof as in the Case 2, we have $x_{n-2}(u/x_1) \in I$.\\
$(C)$: This follows by a similar argument of the Case 2 and the part (A).\\

{\bf Second case:} Consider $r_i > s_i $ for some $i=n-2,n-1,n$. Without loss of generality, we assume that $r_{n-2} > s_{n-2}$. If there exists $r_{j} < s_{j}$ for some $j= 1 , . . . , n-3$, then $x_{j}(u/x_{n-1}) \in I$. Since none of the left-hand sides of the inequalities above change we have the result. Otherwise, we have $r_{j} \geq s_{j}$ for all $j= 1, . . . , n-3$. Since $u$ and $v$ have the same degree, there exists some $r_j < s_j$  with $n-1 \leq j \leq n$. Without loss of generality, we assume that $r_{n-1} > s_{n-1}$. The monomial $x_{n-1}(u/x_{n-2})$ satisfies in all inequalities
except $(2)$. If $r_1 + . . .  +r_{n-3}+r_{n-2} > d_2$, then $x_{n-1}(u/x_{n-2}) \in I$, since none of the left-hand sides of the inequalities above change. If $r_1 + . . .  +r_{n-3}+r_{n-2} \not > d_2$, then
 \[r_1 + . . .  +r_{n-3}+r_{n-2} = d_2 \leq s_1 + . . .  +s_{n-3}+s_{n-2} \]
 which is a contradiction because by our hypothesis $r_{n-2} \geq s_{n-2}$ and $r_j \geq s_j$ for $j=2, . . . ,n-3$. This completes the proof.
\end{proof}

Sturmfels in \cite{S} by an example showed that there exists a monomial ideal $I$ generated in degree $3$ such that $I$ has a linear resolution, while $I^2$ has no linear resolution (see also \cite{C}). After that Herzog, Hibi and Zheng \cite{HHZ} proved that a monomial ideal $I$ generated in degree $2$ has a linear resolution if and only if each power of $I$ has a linear resolution.\\

Thus the following natural question arises:

\begin{Question}
Let $I$ be a monomial ideal of $R$ generated in degree $d$. If $I^i$ has a linear resolution for all $i=1,2,...,d-1$, then is it always true that $I^i$ has a linear resolution for all $i$?
\end{Question}

By the following example we give a counterexample for Question 1.21.
\begin{Example}
Let $R=K[a,b,c,d,e,f,g]$ be the polynomial ring over a field $K$ and
\[I=(ace, acf, acg, ade, bcd, bfg, cde, cdf, cdg, cef, ceg, cfg, def, deg, dfg, efg) \]
be a monomial ideal of $R$. By using Macaulay2 \cite{GS}, we have $\reg(I)=3 ,\reg(I^2)=6$ but $\reg(I^3)=10$. Thus $I$ and $I^2$ have a linear resolution, while $I^3$ has no a linear resolution.
\end{Example}

Conca and Herzog in \cite{CH} proved that if $I$ and $J$ are two polymatroidal monomial ideals of $R$, then $IJ$ is polymatroidal.

It is natural to ask the following question:

\begin{Question}
If $I$ and $J$ are two monomial ideal of $R$ such that $IJ$ is polymatroidal, then is it always true that $I$ and $J$ are polymatroidal?
\end{Question}

In the sequel we give two counterexamples for Question 1.23.

\begin{Example}
Let $R=K[x_1,x_2]$ be the polynomial ring over a field $K$, $I=(x_1,x_2)$ and $J=(x_1^2,x_2^2)$. Then $IJ$ is polymatroidal while $J$ is not polymatroidal.
\end{Example}

\begin{Example}
Let $R=K[x_1,x_2,x_3]$ be the polynomial ring over a field $K$, $I=(x_1^2,x_2^2,x_1x_3,x_2x_3)$ and $J=(x_1^2,x_3^2,x_1x_2,x_2x_3)$. Then $IJ$ is polymatroidal while $I$ and $J$ are not polymatroidal.
\end{Example}




\begin{thebibliography}{}
\bibitem{BH}
S. Bandari and J. Herzog, {\it Monomial localizations and polymatroidal ideals}, Eur. J. Comb.,
{\bf 34}(2013), 752-763.
\bibitem{C}
A. Conca, {\it Regularity jumps for powers of ideals}, Lect. Notes Pure Appl. Math., {\bf 244}(2006), 21-32.
\bibitem{CH}
A. Conca and J. Herzog, {\it Castelnuovo-Mumford regularity of products of ideals}, Collect. Math., {\bf 54}(2003), 137-152.
\bibitem{E}
D. Eisenbud, {\it Commutative Algebra with a View Towards Algebraic Geometry}, GTM., vol.150, Springer, Berlin, (1995).
\bibitem{GS}
D. R. Grayson and M. E. Stillman, {\it Macaulay 2, a software system for research in algebraic geometry}, Available at
    {http://www.math.uiuc.edu/Macaulay2/}.
\bibitem{HH}
J. Herzog and T. Hibi, {\it Discrete polymatroids}, J. Algebraic Combin., {\bf 16}(2002), 239-268.
\bibitem{HH1}
J. Herzog and T. Hibi, {\it Monomial ideals}, GTM., vol.260, Springer, Berlin, (2011).
\bibitem{HHZ}
J. Herzog, T. Hibi and X. Zheng, {\it Monomial ideals whose powers have a linear resolution}, Math. Scand., {\bf 95}(2004), 23-32.
\bibitem{HRV}
J. Herzog, A. Rauf and M. Vladoiu, {\it The stable set of associated prime ideals of a polymatroidal ideal}, J. Algebraic Combin., {\bf 37}(2013), 289-312.
\bibitem{HT}
J. Herzog and Y. Takayama, {\it Resolutions by mapping cones}, Homology Homotopy Appl., {\bf 4}(2002), 277-294.
\bibitem{HV}
J. Herzog and M. Vladoiu, {\it Monomial ideals with primary components given by powers of monomial prime ideals}, Electron. J. Combin., {\bf 21} (2014), P1.69.
\bibitem{KM}
Sh. Karimi and A. Mafi, {\it On stability properties of powers of polymatroidal ideals}, to appear in Collect. Math.
\bibitem{S}
B. Sturmfels, {\it Four counterexamples in combinatorial algebraic geometry}, J. Algebra, {\bf 230}(2000), 282-294.
\bibitem{T}
T. N. Trung, {\it Stability of associated primes of integral closures of monomial ideals}, J. Comb. Theory Ser. A {\bf 116}(2009), 44-54.

\end{thebibliography}
\end{document}